\documentclass[12pt,reqno]{amsart}
\usepackage{amsfonts}
\usepackage{amssymb,color}
\usepackage{amssymb}
\usepackage{mathrsfs}
\usepackage[colorlinks=true]{hyperref}
\hypersetup{urlcolor=blue, citecolor=blue}

\usepackage[margin=3cm, a4paper]{geometry}

\newcommand{\R}{\ensuremath{\mathbb{R}}}

\newcommand{\no}{\nonumber}

\newcommand{\Del}[1]{}

\numberwithin{equation}{section}

\newtheorem{thm}{Theorem}[section]

\newtheorem{lem}[thm]{Lemma}
\newtheorem{prop}[thm]{Proposition}

\theoremstyle{remark}
\newtheorem{rem}{Remark}

\theoremstyle{remark}

\theoremstyle{definition}

\begin{document}

\title[Euler equations in Triebel-Lizorkin spaces]{Remarks on the well-posedness of the Euler equations in the Triebel-Lizorkin spaces}
\author[Z. Guo]{Zihua Guo}
\address{School of Mathematical Sciences, Monash University, Clayton VIC 3800, Australia}
\email{zihua.guo@monash.edu}

\author[K. Li]{Kuijie Li\,$^*$}
\address{School of Mathematical Sciences, Monash University, Clayton VIC 3800, Australia}
\email{likjie@163.com}
\thanks{$^*$~Corresponding author}

\begin{abstract}
We prove the continuous dependence of the solution maps for the Euler equations in the (critical) Triebel-Lizorkin spaces, which was not shown in the previous works \cite{Ch02, Ch03, ChMiZh10}.    The proof relies on the classical Bona-Smith method as \cite{GuLiYi18}, where similar result was obtained in critical Besov spaces $B^1_{\infty,1}$.
\end{abstract}

\maketitle


\section{Introduction}  \label{intro}
This article addresses the ideal incompressible Euler equations in $\R^d$, $d\geq 2$:
\begin{align} 
\partial_t u + u\cdot \nabla u + \nabla p  = 0, \ \ \nabla\cdot u = 0, \ \  u(0,x)= u_0(x),
\end{align}
where $u :\R^+ \times \R^d \to \R^d$ represents the velocity vector, $p$ is scalar pressure, $u_0$ is the initial condition verifying $\nabla\cdot u_0 = 0$.

There are extensive literatures on the mathematical analysis of the Euler equations. The Cauchy problem in very general functional setting has been well studied. Kato~\cite{Ka72} constructed a unique local in time regular solution to the 3D Euler equation with initial data in $H^{m}(\R^3),\, m\geq 3$. Similar result was obtained for initial data belonging to $H^s_{p}(\R^d)$ with $s>1+ d/p$, see~\cite{KaPo88}.  Later, Vishik~\cite{Vi98,Vi99} proved the global existence and uniqueness for 2D Euler equations in the borderline Besov spaces $B^{1+2/p}_{p,1}$ with $1<p<\infty$. Local existence and uniqueness was then extended to critical Besov space $B^{1}_{\infty,1}(\R^d),\,d\geq 2$ by Pak and Park~\cite{PaPa04}, see also~\cite{Ch04} for a systematic treatment in Besov spaces.  
Recently, in~\cite{BoLi15,BoLi152}, Bourgain and Li proved a strongly ill-posedness result for the 2D or 3D Euler equations associated with initial data in Besov space $B^{d/p+1}_{p,q}$ for $1 \leq p < \infty, 1<q\leq \infty$ or Sobolev space $W^{d/p+1,p}$ with $1\leq p<\infty$. For Euler equations, Himonas and Misio{\l}ek~\cite{HiMi10} proved the non-uniform dependence of the solution maps in $H^s(\R^d)$ with $s>0$. So one can only expect continuous dependence. Indeed, the continuous dependence in the Besov space, in particular $B^{1}_{\infty,1}(\R^d)$, was shown recently in \cite{GuLiYi18} using Bona-Smith method (\cite{BoSm75}). 

The existence and uniqueness of the Euler equations in general Triebel-Lizorkin spaces was studied by Chae, first in the subcritical space \cite{Ch02}, and then in the critical space (\cite{Ch03}) $F^{d+1}_{1,q}(\R^d)$ for $1\leq q\leq \infty$. { It is worth noting that a gap in \cite{Ch02} (on the trajectory mapping) was filled by \cite{ChMiZh10}.  It seems to us that the proof of a crucial proposition (Proposition 2.1) and commutator estimate (3.9)  in \cite{Ch03} also have gaps.} The main problem is that the critical space is now $L^1$-based  for which usual technique may fail.  For example, the (vector-valued) Hardy-Littlewood maximal operator used in \cite{Ch02, ChMiZh10} is not $L^1$ bounded.  
On the other hand, in \cite{Takada},  some counter-examples of commutator estimates  in the Besov and the Triebel-Lizorkin spaces were constructed.  In particular, 
\begin{align} \label{countertri}
\|2^{js} [u,\Delta_j]\cdot v \|_{L_x^p l_j^q(\mathbb{Z})} \leq C \|u\|_{F^s_{p,q}} \|v\|_{F^s_{p,q}}, \ \ \  {\rm div}\, u = 0,
\end{align}
fails for $1 \leq p < \infty,\,  1\leq q \leq \infty$,\, $s< 1+d/p$.  However, we shall prove~\eqref{countertri}  holds with $p =1,\,  1 \leq q \leq \infty,\, s= d+1$, see~Proposition~\ref{esticomm2}.


The purpose of this paper is twofold.  First, we fill the gap in \cite{Ch03} and prove relevant estimates in the endpoint Triebel-Lizorkin spaces $F^s_{p,q}$ with $p=1$. 
To do this, we used some new techniques regarding maximal function estimates from \cite{Tr83}.  Second, we show the continuous dependence in the (critical) Triebel-Lizorkin space using Bona-Smith method as in \cite{GuLiYi18}.  This together with the previous results \cite{Ch02, Ch03, ChMiZh10} implies the well-posedness of the Euler equations in these spaces in the sense of Hadamard.
The main result of this paper is

\begin{thm} \label{mainresult}
Assume  that  $d \geq 2$, $(s, p, q)$  satisfies 
\begin{align}
s> \frac{d}{p} + 1,\ (p,q)\in (1,\infty) \times (1,\infty)\ \quad  \textrm{or} \ \quad s\geq d+1,\  p=1,\  q\in [1,\infty).
\end{align}
Then for arbitrary $R>0$, $u_0 \in D(R):=\{\phi \in F^s_{p,q} : \|\phi\|_{F^s_{p,q}} \leq R,\ {\rm{div}}\,\phi = 0\}$,  there exist some $T=T(R,s,p,q,d)>0$ and a unique solution $u:=S_{T}(u_0)\in C([0,T]; F^s_{p,q})$ to the Euler equations. Moreover, it satisfies
\begin{itemize}
   \item[(1)] (Boundedness): there exists some $C=C(s,p,q, d)$, such that 
    \begin{align}
     \|S_{T}(u_0)\|_{L^{\infty}_{T}F^s_{p,q}} \leq C \|u_0\|_{F^s_{p,q}}.
    \end{align}
   \item[(2)] (Continuous dependence): the solution map $u_0 \to S_{T}(u_0)$ is continuous from $D(R)$ to $C([0,T];F^s_{p,q})$. Precisely, for any $\epsilon>0$, there exists $\eta=\eta(u_0, R, s, p, q, d)$ such that for any $\psi \in D(R)$ with $\|\psi - u_0\|_{F^s_{p,q}}<\eta $, then 
   \begin{align}
    \|S_T(u_0) - S_T(\psi)\|_{L_T^{\infty}F^s_{p,q}} < \epsilon.
   \end{align}
\end{itemize} 
\end{thm}
\begin{rem}
Suppose $u_{0} \in F^s_{p, \infty}(\R^d)$ with  $1<p< \infty,\ s>1+d/p$, one can also construct a unique local in time solution belonging to $L_T^{\infty} F^s_{p,\infty}$ for some $T =T(\|u_0\|_{F^s_{p,\infty}})$. 
Local existence and uniqueness, and part (1) was obtained in \cite{Ch02,Ch03, ChMiZh10}, except the case  $s>d+1, p=1$ and $1 \leq q< \infty$, which seems to be new.  For the convenience of reader and to make the paper more self-contained, we also provide a sketched proof in the appendix.  The part (2) seems not proved before.
\end{rem}

\begin{rem}
We remark that the above theorem also holds for the ideal MHD equations studied in \cite{ChMiZh10}.  The proof for MHD has slight difference.  So our results extend the result of \cite{ChMiZh10} to the critical space. 
\end{rem}

Our proof of Theorem~\ref{mainresult} is conceptually similar to the one in~\cite{GuLiYi18}, but the problem is technically harder.  The main difficulty lies in  establishing a Moser type inequality and a commutator estimate in the case of $p= 1$. See Proposition~\ref{endmoser} and Proposition~\ref{esticomm2} in Section~\ref{mainestimate}. 

Next we clarify some notations being used throughout this paper. $\mathcal{S}$ and $\mathcal{S}'$ denote the set of Schwartz functions and tempered distributions over $\R^d$ respectively. $\mathscr{F} f = \hat{f}$ stands  for the Fourier transform of $f$, and $\mathcal{F}^{-1} f = \check{f}$, the  inverse Fourier transform of $f$. The symbol $C$ denotes a generic constant, which may be different from line to line. The function spaces are all defined over $\R^d$. For simplicity, the domain will often be omitted,  e.g. we use $L^p$ instead of $L^p(\R^d)$ in many places, if not otherwise indicated. $B(x,r)$ means a ball centred at $x$ with radius $r$ and $B(r):= B(0,r)$.

Let us introduce the functional setting of this paper. Suppose $\varphi \in C^{\infty}(\R^d)$ satisfies $0\leq \varphi \leq 1$, $\varphi = 1$ on $B(1/2)$ and $\varphi = 0$ outside $B(1)$. Set $\psi(\xi) = \varphi(\xi/2) - \varphi(\xi)$, we denote  $\psi_j(\xi) = \psi(\xi/2^j) $ and $\varphi_j(\xi) = \varphi(\xi/2^j) $. The frequency localization operator is defined by
\begin{align} 
\Delta_j := (\mathscr{F}^{-1} \psi_j) *, \ \ \ S_{j} = P_{\leq j} : = (\mathscr{F}^{-1}\varphi_j)*,
\end{align}
here $*$ is the convolution operator in $\R^d$. It is easy to see $\Delta_j = S_{j+1} -S_j$.  For $1 \leq p <\infty, 1 \leq q \leq \infty$,  the inhomogeneous Triebel-Lizorkin spaces $F^s_{p,q} = F^s_{p,q}(\R^d)$ is defined by
\begin{align}
F^s_{p,q} := \big \{f\in \mathcal{S}'(\R^d),\ \  \|f\|_{F^s_{p,q}} < \infty  \big\},  \nonumber
\end{align}
where
\begin{align}
\|f\|_{F^s_{p,q}}:= \Big \| \Big(|  P_{\leq 0} f|^q + \sum_{j \geq 0} 2^{jsq} |\Delta_j f|^q \Big)^{\frac{1}{q}} \Big \|_{L^p_x},  \nonumber
\end{align}
with the usual modification when $q = \infty$. Let $\mathcal{S}'\backslash \mathcal{P}$ denote the tempered distribution modulo the polynomials, then 
\begin{align}
\dot{F}^s_{p,q} := \big \{f\in \mathcal{S}'\backslash \mathcal{P},\ \  \|f\|_{\dot{F}^s_{p,q}} < \infty  \big\}, \nonumber
\end{align}
where
\begin{align}
\|f\|_{\dot{F}^s_{p,q}}:= \Big \| \Big( \sum_{j \in \mathbb{Z}} 2^{jsq}|\Delta_j f|^q \Big)^{\frac{1}{q}} \Big \|_{L^p_x}.   \nonumber
\end{align}
We remark that for any $s > 0$,  
\begin{align}
\|f\|_{F^s_{p,q}} \sim  \|f\|_{L^p} + \|f\|_{\dot{F}^s_{p,q}}, \ \ 1\leq p<\infty, \ \ 1\leq q \leq \infty.  \nonumber 
\end{align}
See e.g.~\cite{Tr83,WaHuHaGu11}.  Analogously, for $1 \leq p, q \leq \infty$, we have
\begin{align}
\|f\|_{B^s_{p,q}} := \Big ( \|P_{\leq 0}f\|_{L_x^p}^q + \sum_{j \geq 0} 2^{jsq} \|\Delta_j f\|^q_{L_x^p}   \Big)^{\frac{1}{q}},  \no
\end{align}
and 
\begin{align*}
\|f\|_{\dot{B}^s_{p,q}} := \Big ( \sum_{j \in \mathbb{Z}} 2^{jsq} \|\Delta_j f\|^q_{L_x^p}   \Big)^{\frac{1}{q}}.  
\end{align*}
We refer reader to~\cite{BaChDa11,Tr83,WaHuHaGu11} for more introductions on these function spaces.

The remaining part of this paper is structured as follows. In Section~\ref{mainestimate}, we list some well known results  and prove the key estimates for the proof. Section~\ref{pfmainresult} is devoted to proving Theorem~\ref{mainresult}. Finally, we include an appendix, where local Cauchy theory for Euler equations in Triebel-Lizorkin spaces is given.


\section{Auxiliary results} \label{mainestimate}

In this section, we recall some well-known facts and present several results which will be used in the sequel. 
\begin{lem}
Let $1\leq p_0 <p_1 <\infty$, $1\leq q_0 \leq \infty $ and $s_0 - d/p_0 = s_1 -d/p_1$, then the following continuous embeddings hold:
\begin{align}
 \dot{F}^{s_0}_{p_0,q_0} \hookrightarrow \dot{B}^{s_1}_{p_1,p_0}, \ \ \ {F}^{s_0}_{p_0,q_0} \hookrightarrow {B}^{s_1}_{p_1,p_0}.  \no 
\end{align}
\end{lem} 
For the proof, one can refer to~\cite{Ja77}.  As a simple consequence, we have 
\begin{align}
\|f\|_{L^{\infty}}  \leq  \|f\|_{B^0_{\infty,1}} \leq C \|f\|_{F^d_{1,\infty}}, \ \ \  \|\nabla f\|_{L^{\infty}} \leq  \|f\|_{B^1_{\infty,1}} \leq C \|f\|_{F^{d+1}_{1,\infty}}.
\end{align}
The following is a lifting property of the homogeneous Triebel-Lizorkin spaces, whose proof can be found in~\cite{FrToWe88,Tr83}.
\begin{lem} \label{equvitri}
For any $k\in \mathbb{N}$, $(p,q)\in [1,\infty)\times [1,\infty]$ and $s\in \R$, we have 
\begin{align}
c \|D^k f\|_{\dot{F}^s_{p,q}} \leq \|f\|_{\dot{F}^{s+k}_{p,q}} \leq C\|D^k f\|_{\dot{F}^{s}_{p,q}}  \no 
\end{align}
holds for some constant $c, C$, here $D:= \sqrt{-\Delta}$.
\end{lem}

Following the definition, we have for $s\in \R$, 
\begin{align}
\|f\|_{F^s_{p,q}} \leq \|P_{\leq 0} f\|_{L^p} + \|f\|_{\dot{F}^s_{p,q}}, \ \ 1\leq p<\infty, \ \ 1\leq q \leq \infty.
\end{align}
When  treating Euler equations in $F^s_{p,q}$,   we should take caution to deal with the low frequency estimate in $L^p$ (particularly when $p=1$) spaces for the pressure term,  a kernel property needs to be  exploited(see a different treatment in~\cite{PaPa04}), which reads
\begin{lem} \label{kernelinteg}
Let $m(\xi)$ be the Fourier symbol of operator $P_{\leq 0} (-\Delta )^{-1} \partial_l \partial_k$, $1\leq k, l \leq d$ , 
Then there exists a constant $C$, such that
\[
\|\mathscr{F}^{-1}\big(m(\xi) \xi_i \big)\|_{L^1} \leq C, \ \ \ \ \forall\ 1\leq i \leq d.
\]
\end{lem}

\begin{proof}
Since $\Delta_j P_{\leq 0}(-\Delta)^{-1} \partial_l \partial_l = 0$ if $j \geq 1$,  we have 
\begin{align}
\|\mathscr{F}^{-1}\big(m(\xi) \xi_i \big)\|_{L^1} & \leq \sum_{j \leq 0} \|\mathscr{F}^{-1}\big(m(\xi)\psi_j(\xi)\xi_i \big)\|_{L^1}  \no \\
& \leq  \sum_{j \leq 0}2^j \|\mathscr{F}^{-1}\big(m(2^j \xi) \psi(\xi)\xi_i \big)\|_{L^1}   \label{smallestimate}
\end{align}
While according to  Bernstein multiplier theorem (see~\cite{WaHuHaGu11}, p.7),  
\begin{align} 
\|\mathscr{F}^{-1} \big(m(2^j \xi) \psi(\xi)\xi_i \big)\|_{L^1} \leq C \|m(2^j \xi) \psi(\xi)\xi_i \|_{H^{L}} , \ \  L= [d/2]+1. \no 
\end{align}
As $\textrm{supp}\,\psi \subset \{ \xi \in \R^d: 1/2 \leq |\xi | \leq 2  \}$, by a direct calculation, one can assert that there exists some constant $C$ independent of $j$, verifying 
\begin{align}
 \|m(2^j \xi) \psi(\xi)\xi_i \|_{H^{L} }\leq C \Big( \|m(2^j \xi) \psi(\xi)\xi_i\|_{L^2} + \sum_{|\alpha|=L} \|\partial_{\xi}^{\alpha} \big( m(2^j \xi) \psi(\xi)\xi_i\big) \|_{L^2} \ \Big) \leq C.  \no
\end{align}
This combined with~\eqref{smallestimate} implies the desired result.
\end{proof}

We will also need the Hardy-Littlewood maximal function.  For a locally integrable function $f$ in $\R^d$, the maximal function  $Mf(x)$ is defined by 
\begin{align}
Mf(x) = \sup_{r>0} \frac{1}{|B(x,r)|} \int_{B(x,r)} |f(y)| dy.  \no
\end{align} 
In addition, suppose $\Omega \subset \R^d$ is a compact set, we denote 
\begin{align}
\mathcal{S}^{\Omega} = \big\{ f\in \mathcal{S}(\R^d), \ \textrm{supp}\,\hat{f} \subset \Omega   \big \}, \ \ \ \ L_p^{\Omega} = \big\{ f\in L^p(\R^d), \ \textrm{supp}\,\hat{f} \subset \Omega   \big \}.  \no
\end{align}
Below  we recall a lemma on the pointwise estimate in terms of the maximal function, for the proof, see~\cite{Tr83} p.16.
\begin{lem}\label{scalarmaxiesti}
Let $f\in \mathcal{S}^{B(1)}$, 
$0<r<\infty$,  then there exists some constant $C$, such that 
\begin{align}
\sup_{y\in \R^d} \frac{|f(x-y)|}{1+|y|^{\frac{d}{r}}} \leq C \big[M(|f|^r)(x)\big]^{\frac{1}{r}}.
\end{align} 
\end{lem}

\begin{rem}\label{rescalmaxi}
The above conclusion still holds for $f\in L_p^{B(1)}$, 
see~\cite{Tr83}, p.22. In addition, if $\textrm{supp} \hat{f} \subset B(R)$, one can have 
\begin{align} 
\sup_{y\in \R^d} \frac{|f(x-y)|}{1+|Ry|^{\frac{d}{r}}} \leq C \big[M(|f|^r)(x)\big]^{\frac{1}{r}}.
\end{align} 
here $C$ is independent of $R$. In fact, set $g_{R}(\cdot)= f(R^{-1}\cdot)$, then applying Lemma~\ref{scalarmaxiesti} to $g_R$
yields the desired result.
\end{rem}
Next we recall the well-known pointwise maximal function estimate, see~\cite{St93}.
\begin{lem} \label{radialmaxicon}
Let $g(x)$ be a nonnegative radial decreasing integrable function, suppose $|\psi(x) | \leq g(x)$ almost everywhere and $f \in L_{loc}^1(\R^d)$,  then 
\begin{align*}
|\psi_{\epsilon} * f(x)| \leq C M(f)(x), \ \ \ \ \forall\, \epsilon>0,
\end{align*} 
where $\psi_{\epsilon}(x) = \epsilon^{-d} \psi(\epsilon^{-1}x),  C= \|\psi\|_{L^1}$.
\end{lem}

\begin{prop} \label{keyesti}
Let $L>0$, $j,k\in \mathbb{Z}$, $j>k-L$ and $r\in (0,\infty)$. $\psi \in C^{\infty}(\R^d)$  satisfies
\begin{align} \label{decayhy}
 |\psi(z)|(1+|z|^{\frac{d}{r}})  \leq g(z),
\end{align}
where $g(z)$ is some nonnegative radial decreasing integrable function. 
Denote $\psi_{k}(x) = 2^{kd} \psi(2^k x)$, then for any $\theta \in (0,1]$, there exists a constant $C$ independent of $j,k$,  such that the following inequality
\begin{align} \label{convomaxi}
|(\psi_k * f)(x)| \leq C 2^{(j-k)\theta\frac{d}{r}} M(|f|^{1-\theta}) (x) \big[M(|f|^r)(x)\big]^{\frac{\theta}{r}}
\end{align}
holds for all $f\in L_p^{B(c2^j)}$ with $p\geq 1$ and some generic constant $c$.
\end{prop}
\begin{proof}
Consider $f\in \mathcal{S}^{B(c2^j)}$ first,  we have 
\begin{align*}
|(\psi_k*f)(x)| & \leq \int_{\R^d} |\psi(y)| |f(x-2^{-k}y) | dy  \\
                       & \leq \Big[\int_{\R^d} |\psi(y)| |f(x-2^{-k}y)|^{1-\theta} (1+(2^{j-k}|y|)^{\theta d/r}) dy\Big]   \no \\
  & \qquad \qquad \qquad \qquad \qquad \quad \qquad \times    \sup_{y\in \R^d} \frac{|f(x-2^{-k}y)|^{\theta}}{1+(2^{j-k}|y|)^{\theta d/r}}.
\end{align*}
In view of~\eqref{rescalmaxi}, one can see
\[
\sup_{y\in \R^d} \frac{|f(x-2^{-k}y)|^{\theta}}{1+(2^{j-k}|y|)^{\theta d/r}} 
\leq C[ M(|f|^r)(x)]^{\frac{\theta}{r}}.
\]
Given that $j- k> -L$, then 
\begin{align*}
& \int_{\R^d}  |\psi(y)| |f(x-2^{-k}y)|^{1-\theta} (1+(2^{j-k}|y|)^{\theta d/r}) dy  \\
 & \leq C  2^{(j-k)\frac{\theta d}{r}}\int_{\R^d}  |\psi(y)|  |f(x-2^{-k}y)|^{1-\theta} (1+|y|^{\theta d/r}) dy \no \\
 &  \leq C 2^{(j-k)\frac{\theta d}{r}} M(|f|^{1-\theta})(x).
\end{align*}
where we used hypothesis \eqref{decayhy} and Lemma~\ref{radialmaxicon} in the last inequality.  Hence, the proof is completed  for $f\in \mathcal{S}^{B(c2^j)}$.

 For general $f\in L^{B(c2^j)}_p$, one can choose $\varphi \in \mathcal{S}$, such that $\varphi(0) = 1$ and $\textrm{supp} \hat{\varphi} \subset B(1)$.  Denote  $f_{\delta}:= \varphi(\delta x) f(x)$,   applying previous result to $f_{\delta}$ and letting $\delta \to 0$,  one can find~\eqref{convomaxi} follows, see also~\cite{Tr83} (p.22) for more explanations. 
\end{proof}
\begin{rem} \label{estimatelow}
One can easily see from the above proof 
\begin{align} 
|(\psi_k *f)(x)| \leq C_L  \big[M(|f|^r)(x)\big]^{\frac{1}{r}},
\end{align}
provided $j \leq k+L$. Then for $1 \leq p< \infty, 1\leq q \leq \infty$, it follows 
\begin{align}
\|P_{\leq m} f\|_{{F}^{s+l}_{p,q}}  \leq C 2^{ml}\|f\|_{{F}^{s}_{p,q}}, \ \  \forall\, m,\ l \geq 0.
\end{align}
\end{rem}

The following vector-valued maximal function estimate will also be frequently used, see~\cite{FeSt71,St93} for a proof.
\begin{prop} \label{maxiesti}
Let $(p,q)\in (1,\infty)\times (1,\infty] $ or $p=q=\infty$ be given. Suppose $ \{f_j\}_{j\in \mathbb{Z}}$ is a sequence of functions in $L^p(\R^d)$ satisfying $\|f_j\|_{l_j^q(\mathbb{Z})} \in L^p(\R^d)$,  then 
\begin{align}
\|M(f_j)(x)\|_{L_x^pl_j^q} \leq C \|f_j(x)\|_{L_x^pl_j^q}   \no
\end{align}
for some constant $C=C(p,q)$.
\end{prop}

Next we establish the Moser type inequality for the Triebel-Lizorkin spaces. First we recall
\begin{prop}[\cite{Ch02}] \label{nonendmoser}
Let $(p,q)\in (1,\infty)\times (1,\infty]$ or $p = q = \infty$, $s>0$. There exists some positive constant  $C$ with the following property:
\begin{align}
\|fg\|_{\dot{F}^s_{p,q}} \leq C(\|f\|_{L^{\infty}} \|g\|_{\dot{F}^s_{p,q}} + \|g\|_{L^{\infty}} \|f\|_{\dot{F}^s_{p,q}}).  \no 
\end{align} 
\end{prop}

Proposition 2.1 in \cite{Ch03} claimed that the above proposition also holds for $s>0, p=1, 1\leq q\leq \infty$.  However, the proof of Proposition 2.1  in \cite{Ch03} seems to have gaps.  In the following proposition, we re-prove the endpoint case $p =1$, which exactly complements the nonendpoint  conuterpart. 

\begin{prop}[Endpoint case] \label{endmoser}
Let $q\in  [1,\infty]$ be given,  then there exists some constant $C$ such that 
\begin{align} \label{prod1}
\|fg\|_{\dot{F}^s_{1,q}} \leq C(\|f\|_{L^{\infty}} \|g\|_{\dot{F}^s_{1,q}} + \|f\|_{\dot{F}^s_{1,q}} \|g\|_{L^{\infty}}), \ \ \ s>0, 
\end{align}
holds for scalar functions $f$ and $g$. Additionally, suppose that $v$ is a scalar function and $u$ is a vector-valued function with ${\rm{div}}\,u=0$, then 
\begin{align} \label{prod2}
\|u\cdot \nabla v\|_{\dot{F}^s_{1,q}} \leq C( \| u\|_{L^{\infty}} \|\nabla v\|_{\dot{F}^s_{1,q}} + \|\nabla v\|_{L^{\infty}}\|u\|_{\dot{F}^s_{1,q}}), \ \ \ s>-1.
\end{align}
and 
\begin{align} \label{prod3}
\|u\cdot \nabla v\|_{\dot{F}^s_{1,q}} \leq C( \| u\|_{L^{\infty}} \|\nabla v\|_{\dot{F}^s_{1,q}} + \|v\|_{L^{\infty}}\|\nabla u\|_{\dot{F}^s_{1,q}}), \ \ \ s>-1.
\end{align}
\end{prop}

\begin{proof}
We use the following Bony decomposition  (\cite{Bo81})
\begin{align}
fg = T_fg + T_gf+ R(f,g),  \no 
\end{align}
where 
$$
T_f g = \sum S_{j-3} f \Delta_j g = \sum_{j\in \mathbb{Z}} \sum_{l \leq j-4} \Delta_l f \Delta_j g, \quad R(f,g)=\sum_{|j-k|\leq 3} \Delta_j f \Delta_k g.
$$	
Due to frequency interaction, one can figure out that $\Delta_m (S_{j-3} f\Delta_j g) = 0$ if $|j-m| \geq 3$, hence for any $0<r<\infty$, 
\begin{align}
\big |2^{ms} \Delta_m T_fg \big | & = \Big |2^{ms} \sum_{|j-m| \leq 2} \Delta_m (S_{j-3} f\Delta_j g)\Big |  \no \\
& \leq C\sum_{|a|\leq 2}2^{ms} \Big[M(|S_{m+a-3}f\Delta_{m+a}g|^{r})(x)\Big]^{\frac{1}{r}},  \no 
\end{align}
where we used Proposition~\ref{keyesti} with $\theta =1$.  As such, choosing $0<r<1$ and applying Proposition~\ref{maxiesti}, we have
\begin{align}
\|T_f g\|_{\dot{F}^s_{1,q}} & \leq C\sum_{|a|\leq 2} \|M(|(S_{m+a-3}f) (2^{ms}\Delta_{m+a}g)|^{r})(x)\|^{\frac{1}{r}}_{L_x^{1/r}l_{m}^{q/r}}  \no \\
& \leq C\sum_{|a|\leq 2} \|2^{ms}|S_{m+a-3}f\Delta_{m+a}g|(x)\|_{L_x^{1}l_{m}^{q}} \no \\
& \leq C\|f\|_{L^{\infty}} \|g\|_{\dot F^s_{1,q}}.  \no
\end{align}
Similarly, 
\begin{align}
\|T_g f\|_{\dot{F}^s_{1,q}} \leq C \|g\|_{L^{\infty}} \|f\|_{\dot{F}^s_{1,q}}.  \no 
\end{align}
Now we estimate $R(f,g)= \sum_{|b|\leq 3} \sum_{j\in \mathbb{Z}} \Delta_j f\Delta_{j+b} g$. For arbitrary fixed $r \in (0,1)$, as $s>0$, we can specify $\theta \in (0,1)$ such that $s> d \theta/r $.  Using the property of frequency support, one can assert that there exists a constant $L$, such that 
\begin{align}  \label{estimate11}
\|R(f,g)\|_{\dot{F}^s_{1,q}} & \leq \sum_{|b|\leq 3}\Big \| \sum_{j>m-L} 2^{ms} \Delta_m(\Delta_j f \Delta_{j+b}g) \Big\|_{L_x^1 l_m^q} \no  \\
& \leq C\sum_{|b|\leq 3} \Big \| \sum_{j>m-L} 2^{(m-j)(s-\theta d/r)} M(|2^{js}\Delta_j f \Delta_{j+b}g|^{1-\theta})(x)|)  \no \\ 
& \qquad \qquad \qquad \qquad  \qquad \qquad  \times \big[M(|2^{js}\Delta_j f \Delta_{j+b}g|^{r})(x) \big]^{\frac{\theta}{r}} \Big \|_{L_x^1l_{m}^q}  \no \\
& \leq C\sum_{|b|\leq 3} \Big \| M(|2^{js}\Delta_j f \Delta_{j+b}g|^{1-\theta})(x)|)  \big[M(|2^{js}\Delta_j f \Delta_{j+b}g|^{r})(x) \big]^{\frac{\theta}{r}} \Big \|_{L_x^1l_{j}^q}  \no \\
&\leq  C\sum_{|b|\leq 3}     \big\| M(|2^{js}\Delta_j f \Delta_{j+b}g|^{1-\theta})(x)|) \big\|_{L_x^{\frac{1}{1-\theta}} l_j^{\frac{q}{1-\theta}}}   \no \\
& \qquad \qquad \qquad \qquad \qquad \qquad  \times \Big \| \big[M(|2^{js}\Delta_j f \Delta_{j+b}g|^{r})(x) \big]^{\frac{\theta}{r}}  \Big\|_{L_x^{\frac{1}{\theta}} l_j^{\frac{q}{\theta}}}         \no   \\
& \leq  C \sum_{|b|\leq 3} \big\| |2^{js}\Delta_j f \Delta_{j+b}g|^{1-\theta})(x)|\big\|_{L_x^{\frac{1}{1-\theta}} l_j^{\frac{q}{1-\theta}}}    \Big \| M(|2^{js}\Delta_j f \Delta_{j+b}g|^{r})(x)  \Big\|^{\frac{\theta}{r}}_{L_x^{\frac{1}{r}} l_j^{\frac{q}{r}}} \no \\
& \leq C \|f\|_{L^{\infty}} \|g\|_{\dot{F}^s_{1,q}}.  
\end{align}
where we utilized Proposition~\ref{keyesti}, Young's inequality and Proposition~\ref{maxiesti} from the second to the last inequality. This yields~\eqref{prod1}.

As to the proof of~\eqref{prod2}, we first note that $ u\cdot \nabla v = u^{l} \partial_{l} v $, here summation over repeated indices is adopted. Similarly, 
\begin{align}
u \cdot \nabla v = T_{u^l} \partial_l v + T_{\partial_l v} u^l + R(u^l, \partial_l v).   \no 
\end{align}
In view of the argument above, one can easily see 
\begin{align}
\|T_{u^l} \partial_l v\|_{\dot{F}^s_{1,q}} + \|T_{\partial_l v} u^{l}\|_{\dot{F}^s_{1,q}} \leq C(\|u\|_{L^{\infty}} \|\nabla v\|_{\dot{F}^s_{1,q}} + \|\nabla v\|_{L^{\infty}} \|u\|_{\dot{F}^s_{1,q}} ).  \no 
\end{align}
Thanks to the divergence free condition on $u$, we know $R(u^l,\partial_l v) = \partial_l R(u^l, v)$, by Lemma~\ref{equvitri}
\begin{align}
\|R(u^l,\partial_l v)\|_{\dot{F}^s_{1,q}} \leq C \sum_{1\leq l \leq d} \|R(u^l, v)\|_{\dot{F}^{s+1}_{1,q}}.  \no 
\end{align}
Then the argument of $R(f,g)$ above implies that
\begin{align}
\|R(u^l, v)\|_{\dot{F}^{s+1}_{1,q}} \leq C \|u^l\|_{L^{\infty}} \|\nabla v\|_{\dot{F}^s_{1,q}}   \no 
\end{align}
holds for all $s>-1$. Hence, \eqref{prod2} is proved.  Finally, owing to Lemma~\ref{equvitri}, 
\begin{align}
\|u\cdot\nabla v\|_{\dot{F}^s_{1,q}} \leq C \sum_{1\leq l \leq d}  \|u^l v\|_{\dot{F}^{s+1}_{1,q}}.  \no 
\end{align}
Thus~\eqref{prod3} is a simple consequence of~\eqref{prod1}.
\end{proof}

We shall conclude this section by presenting the commutator estimates, which turns out to be an important tool in~\cite{KaPo88}. In order to estimate the $F^s_{p,q}$ norm of the solution to the Euler equations, a commutator involve frequency localization operator occurs naturally, Let us first recall  that 
\begin{align}
[f,\Delta_j]g : = f\Delta_j g - \Delta_j (fg). \no 
\end{align}
\begin{prop}[\cite{ChMiZh10}] \label{esticomm1}
Let $(p,q)\in (1,\infty)\times (1,\infty]$. Suppose $f$ is a divergence free vector field, then there exists a constant $C$, such that for $s>0$,
\begin{align}
\|2^{js}[f,\Delta_j]\cdot \nabla g\|_{L_x^p l_j^q} \leq C(\|\nabla f\|_{L^{\infty}}\|g\|_{\dot{F}^s_{p,q}} + \|\nabla g\|_{L^{\infty}} \|f\|_{\dot{F}^s_{p,q}}  ).  \no 
\end{align} 
or for $s>-1$, 
\begin{align}
\|2^{js}[f,\Delta_j]\cdot \nabla g\|_{L_x^p l_j^q} \leq C(\|\nabla f\|_{L^{\infty}}\|g\|_{\dot{F}^s_{p,q}} + \|g \|_{L^{\infty}} \|\nabla f\|_{\dot{F}^s_{p,q}}  ).  \no 
\end{align}
\end{prop}

\begin{prop}[Endpoint commutator estimate] \label{esticomm2}
Let $d \geq 1$ denote the space dimension, $q\in [1,\infty]$ be given.  There exists a constant $C$, such that
\begin{align} \label{esti1}
\|2^{js}[f,\Delta_j]\cdot \nabla g\|_{L_x^1 l_j^q} \leq C(\|\nabla f\|_{L^{\infty}}\|g\|_{\dot{F}^s_{1,q}} + \|\nabla g\|_{L^{\infty}} \|f\|_{\dot{F}^s_{1,q}}  ), \ \ \ s>0,
\end{align}
and 
\begin{align}\label{esti2}
\|2^{js}[f,\Delta_j]\cdot \nabla g\|_{L_x^1 l_j^q} \leq C(\|\nabla f\|_{L^{\infty}}\|g\|_{\dot{F}^s_{1,q}} + \|g \|_{L^{\infty}} \|\nabla f\|_{\dot{F}^s_{1,q}}  ), \ \  \ s>-1,
\end{align}
hold for all scalar function $g$ and  vector-valued function $f$ with  ${\rm{div}} f = 0$.
\end{prop}

\begin{proof}
We first show~\eqref{esti1}. Let $f=(f^l)_{1\leq l\leq d}$, according to Bony decomposition, one can see
\begin{align}
[f,\Delta_j]\cdot \nabla g & = f^l \Delta_j \partial_l g - \Delta_j(f^l\partial_l g) \no\\
&= T_{\Delta_j \partial_l g} f^l + R(f^l,\Delta_j \partial_l g) + [T_{f^l}, \Delta_j]\partial_l g - \Delta_j T_{\partial_l g} f^l - \Delta_j R(f^l,\partial_l g). \no
\end{align}
It suffices to bound the above five terms in turn. Note that $S_{m-3}\Delta_j = 0$ if $m \leq j+2$, thus 
\begin{align}
\|2^{js} T_{\Delta_j \partial_l g} f^l\|_{L_x^1l_j^q} & = \Big \|\sum_{m>j+2} 2^{(j-m)s} (S_{m-3}\Delta_j \partial_l g)(2^{ms} \Delta_m f^l) \Big\|_{L_x^1l_j^q}  \no \\
& \leq C \|\nabla g\|_{L^{\infty}} \sum_{1\leq l \leq d}   \Big \|\sum_{m>j+2} 2^{(j-m)s} 2^{ms} |\Delta_m f^l|\Big \|_{L_x^1l_j^q} \no \\
& \leq  C \|\nabla g\|_{L^{\infty}} \|f\|_{\dot{F}^s_{1,q}}.  \no 
\end{align}
where we used Young's inequality in the last step as $s>0$. On the estimate of $R(f^l, \Delta_j \partial_l g)$, 
one can see
\begin{align}
\|2^{js} R(f^l, \Delta_j \partial_l g)\|_{L_x^1l_j^q} \leq \sum_{|a|\leq 3} \sum_{|j-m|\leq 1} \big \|2^{js} \partial_l \big[(\Delta_{m+a} f^l)\Delta_m (\Delta_jg) \big]  \big \|_{L_x^1l_j^q}.  
\end{align}
Let $\phi \in C_0^{\infty} (\R^d)$, $\phi = 1$ on $B(2^5)$. Due to the fact 
$$\textrm{supp} \mathscr{F} \big( \Delta_{m+a} f^l \Delta_m \Delta_jg \big) \subset B(2^{m+5}). $$
We can assert
\begin{align*}
\partial_l \big( \Delta_{m+a} f^l \Delta_m \Delta_jg \big) = \check{\phi}_{m} * \partial_l \big( \Delta_{m+a} f^l \Delta_m \Delta_jg \big), \ \ \ \check{\phi}_{m}(\cdot) := 2^{md}\check{\phi}(2^m \cdot).
\end{align*} 
Now applying Proposition~\ref{keyesti} with $\theta =1$ and Proposition~\ref{maxiesti}, we get
\begin{align}
& \|2^{js} R(f^l, \Delta_j \partial_l g)\|_{L_x^1l_j^q} \\
& \leq C\sum_{1\leq l \leq d} \sum_{|a|\leq 3} \sum_{|b|\leq 1} \big\| \big[M(|2^j\Delta_{j+a+b}f^l||\Delta_{j+b}(2^{js}\Delta_j g)|)^{r} (x) \big]^{\frac{1}{r}} \big\|_{L_x^1l_j^q}  \no \\
& \leq C \|\nabla f\|_{L^{\infty}} \|g\|_{\dot{F}^s_{1,q}}.  \no 
\end{align}
where we used the following simple fact 
$$2^j\|\Delta_{j+b} f\|_{L^{\infty}} \leq C_b \|\nabla f\|_{L^{\infty}}.$$ 
Concerning the third term~$[T_{f^l},\Delta_j]\partial_l g$, we first note that
\begin{align}
[T_{f^l},\Delta_j]\partial_l g = \sum_{m \in \mathbb{Z}}[S_{m-3}f^l,\Delta_j]\Delta_m \partial_l g =\sum_{|m-j|\leq 2}[S_{m-3}f^l,\Delta_j]\Delta_m \partial_l g.  \no 
\end{align}
Furthermore,
\begin{align}
&\sum_{|m-j|\leq 2}\big |[S_{m-3}f^l,\Delta_j]\Delta_m \partial_l g \big|   \no \\
&=\sum_{|a|\leq 2} \Big|  \int_{\R^d} 2^{jd} \check{\psi}(2^j(x-y))(S_{j+a-3}f^l(x) - S_{j+a-3}f^l(y)) \Delta_{j+a}\partial_l g(y)dy \Big |  \no \\ 
&\leq \sum_{|a|\leq 2} \int_{\R^d} 2^{j(d+1)} \big| (\partial_l \check{\psi})(2^j(x-y))(S_{j+a-3}f^l(x) - S_{j+a-3}f^l(y)) \Delta_{j+a} g(y)\big |dy  \no \\
& \leq C\sum_{|a|\leq 2}\big\|\nabla S_{j+a-3} f^l\big\|_{L^{\infty}}\int_{\R^d} 2^{j(d+1)}\big|x-y\big| \big|(\partial_l\check{\psi})(2^j(x-y)) \big|  \big|\Delta_{j+a} g(y) \big| dy \no  \\
& \leq C\sum_{|a|\leq 2}  \big\|\nabla f \big\|_{L^{\infty}} \big[M(|\Delta_{j+a}g|^r)(x)\big]^{\frac{1}{r}}. \no 
\end{align}
here we used the ${\rm div}\,f = 0$, mean value theorem and Proposition~\ref{keyesti} with $\theta =1$ from the second to the fourth step.  Therefore, it follows from Proposition~\ref{maxiesti} that
\begin{align}
\|2^{js}[T_{f^l},\Delta_j]\partial_l g\|_{L_x^1l_j^q} 
 &\leq C\sum_{|a|\leq 2} \big\|\nabla f \big\|_{L^{\infty}}  \Big \|\big[M(|2^{js}\Delta_{j+a}g|^r)(x)\big]^{\frac{1}{r}}  \Big \|_{L_x^1l_j^q}  \no \\
& \leq C \big\|\nabla f \big\|_{L^{\infty}}  \|g\|_{\dot{F}^s_{1,q}}.   \no 
\end{align}
Regarding the term $\Delta_j T_{\partial_l g}f^l$, applying Proposition~\ref{keyesti},  we have 
\begin{align}
\|2^{js}\Delta_j T_{\partial_lg}f^l\|_{L_x^1l_j^q} & = \Big \|2^{js}\sum_{|m-j|\leq 2}\Delta_j  \big( (S_{m-3}\partial_l g) \Delta_m f^l \big) \Big\|_{L_x^1l_j^q}  \no \\
& \leq C \sum_{|a|\leq 2} \Big \| \big[M\big( \big|(S_{j+a-3}\partial_l g) (2^{js}\Delta_{j+a} f^l )\big|^{r}(x) \big]^{\frac{1}{r}} \Big\|_{L_x^1l_j^q}   \no \\
& \leq C \sum_{|a|\leq 2} \sum_{1\leq l\leq d} \big\|\nabla g  \big\|_{L^{\infty}} \big\| 2^{js} \Delta_{j+a}  f^l\big\|_{L_x^1 l_j^q}  \no \\
&\leq C\|\nabla g\|_{L^{\infty}} \|f\|_{\dot{F}^s_{1,q}}. \no 
\end{align}
Finally, we estimate the term $\Delta_j R(f^l, \partial_l g)$.  Since $s>-1$, for arbitrary $r\in (0,1)$, one can select $\theta \in (0,1)$ small enough, such that $s+1 > d \theta /r$.  Due to  frequency interaction,   one can observe  that there exists a constant $L$, such that 
\begin{align}
\big |\Delta_j R(f^l, \partial_l g)\big|&= \Big| \sum_{|a|\leq 3}\sum_{m \geq j-L} \Delta_j \partial_l \big( \Delta_m f^l \Delta_{m+a}g \big) \Big|  \no \\
& \leq \sum_{1\leq l\leq d} \sum_{|a|\leq 3}\sum_{m \geq j-L}C 2^j 2^{(m-j)\frac{d\theta}{r}}  M(|\Delta_m f^l \Delta_{m+a}g|^{1-\theta})  \no \\
& \qquad  \qquad \qquad \qquad \qquad  \qquad \qquad \quad \times \big[M(|\Delta_m f^l \Delta_{m+a}g|^r) (x) \big]^{\frac{\theta}{r}}.  \no 
\end{align}
where  Proposition~\ref{keyesti}  is used. Thanks to Young's inequality,  one can get
\begin{align}
& \|2^{js} \Delta_j R(f^l, \partial_l g)\|_{L_x^1l_j^q}  \no \\
& \leq C \sum_{1 \leq l \leq d}\sum_{|a|\leq 3} \bigg\|\sum_{m\geq j-L}2^{(j-m)(s+1-d\theta/r)}  M(|(2^m\Delta_m f^l) (2^{ms}\Delta_{m+a}g)|^{1-\theta})   \no \\
 & \qquad \qquad \qquad \qquad \qquad  \qquad   \times \big[M(|(2^m\Delta_m f^l) (2^{ms}\Delta_{m+a}g)|^r) (x)\big]^{\frac{\theta}{r}} \bigg\|_{L_x^1l_j^q}  \no \\
& \leq C \sum_{1 \leq l \leq d}\sum_{|a|\leq 3} \bigg\| M(|(2^m\Delta_m f^l)(2^{ms} \Delta_{m+a}g)|^{1-\theta}) \no \\
&  \qquad  \qquad \qquad \qquad \qquad  \times  \big[M(|(2^m\Delta_m f^l) (2^{ms}\Delta_{m+a}g)|^r) (x)\big]^{\frac{\theta}{r}}  \bigg\|_{L_x^1l_m^q}.  \no 
\end{align}
Then one can argue analogously as~\eqref{estimate11} to obtain 
\begin{align*}
\|2^{js} \Delta_j R(f^l, \partial_l g)\|_{L_x^1l_j^q}  \leq C \|\nabla f\|_{L^{\infty}} \|g\|_{\dot{F}^s_{1,q}}.  
\end{align*}
Gathering the estimates above, we find~\eqref{esti1} follows.  

In order to show~\eqref{esti2}, it suffices to slightly modify the estimate of the terms $T_{\Delta_j \partial_l g}f^l$ and $\Delta_j T_{\partial_l g} f^l$.  Note that $s>-1$, then  
\begin{align}
\|2^{js} T_{\Delta_j \partial_l g} f^l\|_{L_x^1l_j^q} & = \Big\|\sum_{m>j+2} 2^{(j-m)s} (S_{m-3}\Delta_j \partial_l g)(2^{ms} \Delta_m f^l)\Big\|_{L_x^1l_j^q}  \no \\
& \leq C \Big\|\|\Delta_j \partial_l g\|_{L^{\infty}} \sum_{m>j+2}  2^{(j-m)s} |2^{ms} \Delta_m f^l| \Big \|_{L_x^1l_j^q}  \\
& \leq C \|g\|_{L^{\infty}} \sum_{1\leq l \leq d}   \Big \|\sum_{m>j+2} 2^{(s+1)(j-m)} 2^{m(s+1)} |\Delta_m f^l|\Big \|_{L_x^1l_j^q} \no \\
& \leq  C \| g\|_{L^{\infty}} \|f\|_{\dot{F}^{s+1}_{1,q}}. \no 
\end{align}
where we used Young's inequality.  Regarding to the term $\Delta_j T_{\partial_l g} f^l$, thanks to Proposition~\ref{keyesti} and ${\rm div}\,u=0$,  one can immediately have
\begin{align}
\|2^{js}\Delta_j T_{\partial_lg}f^l\|_{L_x^1l_j^q} & = \Big \|2^{js}\sum_{|m-j|\leq 2}\Delta_j \partial_l \big( (S_{m-3} g) (\Delta_m f^l) \big) \Big\|_{L_x^1l_j^q}  \no \\
& \leq C\sum_{|a|\leq 2}\sum_{1\leq l \leq d}  \big \| 2^{(s+1)j} \big[M(|S_{j+a-3}g \Delta_{j+a} f^l|^r(x))  \big]^{\frac{1}{r}}  \big\|_{L_x^1l_j^q}  \no \\
& \leq C \|g\|_{L^{\infty}} \|f\|_{\dot{F}^{s+1}_{1,q}}.  \no 
\end{align}
This completed the proof. 
\end{proof}


\section{Proof of the Main result} \label{pfmainresult}

In this section, we follow the scheme of~\cite{GuLiYi18}  to demonstrate that  the solution map  of Euler equations  is continuous in Triebel-Lizorkin spaces.
\begin{proof} [Proof of Theorem~\ref{mainresult}]
 The proof is divided  into four steps: \\
{\it Step 1. } It follows from local Cauchy theory that there exists some $T=T(\|u_0\|_{F^s_{p,q}})$ and a unique solution $u =S_{T}(u_0) \in C([0,T];F^s_{p,q})$ such that   
\begin{align} \label{bddintri}
\|u\|_{L_{T}^{\infty} F^s_{p,q}} \leq C\|u_0\|_{F^s_{p,q}}.
\end{align}
Moreover,  if $u_0 \in F^{s+\gamma}_{p,q}$ with some $\gamma>0$, then
\begin{align}
\|u\|_{L^{\infty}_T F^{s+\gamma}_{p,q}} \leq C \|u_0\|_{F^{s+\gamma}_{p,q}}.
\end{align}
One can refer to~\cite{Ch02,Ch03} or  Theorem~\ref{lwp} in the Appendix for more details.

{\it Step 2.} For any $u_0, v_0 \in D(R)=\{ \psi \in F^s_{p,q}: \, {\rm div}\, \psi = 0, \ \|\psi\|_{F^s_{p,q}} \leq R\}$, we have
\begin{align} \label{differ}
\|S_T(u_0) - S_T(v_0)\|_{L_T^{\infty}F^{s-1}_{p,q}} \leq C\|u_0 - v_0\|_{F^{s-1}_{p,q}}.
\end{align}
In fact, let $u =S_T(u_0),\, v= S_T(v_0)$. Set $w= u - v$, then $w$ solves
\begin{equation*}
\begin{cases}
\partial_t w + w\cdot \nabla u + v \cdot \nabla w + \nabla ( P(u) -P(v)) = 0, & \\
{\rm div}\, w = 0, & \\ 
w(0,x) = w_0= u_0 -v_0. &
\end{cases}
\end{equation*} 
here $P(u):= (-\Delta)^{-1} {\rm div}(u\cdot \nabla u)$. Applying the frequency localization operator $\Delta_j $, one can find
\begin{align}
\partial_t \Delta_j w + v\cdot \nabla \Delta_j w+ \Delta_j( w\cdot \nabla u ) +  \nabla \Delta_j ( P(u) -P(v)) = [v,\Delta_j]\cdot \nabla w.  \no 
\end{align}
As in~\cite{Ch02,ChMiZh10}, we introduce  particle trajectory mapping $X(\alpha,t)$ defined by the solution of the ordinary differential equations
\begin{equation*}
\begin{cases}
\frac{\partial}{\partial t} X(\alpha,t) = v (X(\alpha,t), t), & \\
X(\alpha,0) = \alpha. &
\end{cases}
\end{equation*}
This implies 
\begin{equation} \label{diffequation}
\begin{split}
\Delta_j w (X(\alpha,t),t) = \Delta_j w_0(\alpha) + \int_0^t \Big([v,\Delta_j]\cdot \nabla w - \Delta_j(w\cdot \nabla u)\Big)(X(\alpha,\tau),\tau)  \\
- \nabla \Delta_j(P(u)-P(v))(X(\alpha, \tau),\tau) d\tau.  
\end{split}
\end{equation}
Note that ${\rm div}\,v =0$,  so $X(\alpha,t)$ is a measure preserving mapping. Multiplying $2^{j(s-1)}$ and taking $L_{\alpha}^pl_j^q$ norm on both sides of~\eqref{diffequation},  we can see
\begin{equation} \label{trieesti}
\begin{split}
\|w(t)\|_{\dot{F}^{s-1}_{p,q}} \leq \|w_0\|_{\dot{F}^{s-1}_{p,q}} + \int_{0}^t \|\nabla(P(u)-P(v))\|_{\dot{F}^{s-1}_{p,q}} + \|w\cdot \nabla u\|_{\dot{F}^{s-1}_{p,q}}  \\
+ \big \|2^{(s-1)j}[v,\Delta_j]\cdot\nabla w \big\| _{L_x^pl_j^q} d\tau. 
\end{split}
\end{equation}
Similarly, 
\begin{align}
w(X(\alpha,t),t) = w_0(\alpha) - \int_0^t \Big(w\cdot \nabla u + \nabla(P(u)-P(v))\Big)(X(\alpha,\tau),\tau) d\tau. \no 
\end{align}
which leads to 
\begin{align} \label{lpesti}
\|w(t)\|_{L^p} \leq \|w_0\|_{L^p} + \int_0^t \|w\cdot \nabla u(\tau)\|_{L^p} + \|\nabla (P(u)-P(v))(\tau)\|_{L^p} d\tau.
\end{align}
Combining the estimates~\eqref{trieesti} and~\eqref{lpesti}, one can get
\begin{equation} \label{loweresti}
\begin{split}
\|w\|_{F^{s-1}_{p,q}} \leq \|w_0\|_{F^{s-1}_{p,q}} + \int_{0}^t \|w\cdot\nabla u\|_{F^{s-1}_{p,q}} + \|\nabla (P(u)-P(v))\|_{F^{s-1}_{p,q}}  \\
+ \big\|2^{j(s-1)}[v,\Delta_j]\cdot\nabla w \big\| _{L_x^pl_j^q} d\tau.
\end{split}
\end{equation}
Recall that for $a\in \R$, we have
\begin{align} \label{controlnonho}
\|f\|_{F^{a}_{p,q}} \leq \|P_{\leq 0} f\|_{L^p} + \|f\|_{\dot{F}^{a}_{p,q}}.
\end{align}
Noticing that
\begin{align}
\nabla (P(u) -P(v)) = \nabla (-\Delta)^{-1} {\rm div}(w\cdot \nabla u + v\cdot \nabla w). \no 
\end{align}
By Lemma~\ref{kernelinteg}, we can assert
\begin{align}
\|P_{\leq 0} \nabla (-\Delta)^{-1}{\rm div} (w\cdot \nabla u)\|_{L^p} \leq C \|w\|_{L^{\infty}} \|u\|_{L^{p}}. \no 
\end{align}
Owing to the boundedness of operator $\partial_j \partial_k(-\Delta)^{-1}$ in $\dot{F}^{s-1}_{p,q}$(see \cite{FrToWe88,St93,Ch03}) and Propositions \ref{nonendmoser} and \ref{endmoser}, one can find
\begin{align}
\|\nabla (-\Delta)^{-1}{\rm div} (w\cdot \nabla u)\|_{\dot{F}^{s-1}_{p,q}} &\leq C \|w\cdot \nabla u\|_{\dot{F}^{s-1}_{p,q}}  \no \\
& \leq C\big( \|w\|_{L^{\infty}} \|u\|_{\dot{F}^s_{p,q}} + \|w\|_{\dot{F}^{s-1}_{p,q}} \|\nabla u\|_{L^{\infty}}\big). \no 
\end{align}
Consequently,
\begin{align}
\|\nabla (-\Delta)^{-1}{\rm div}(w\cdot \nabla u)\|_{F^{s-1}_{p,q}} \leq C\|w\|_{F^{s-1}_{p,q}} \|u\|_{F^s_{p,q}}.  \no 
\end{align}
On the other hand, ${\rm div}(v\cdot \nabla w) = {\rm div}(w\cdot \nabla v)$, we finally obtain
\begin{align}
\|\nabla (P(u)-P(v))\|_{F^{s-1}_{p,q}} \leq C\|w\|_{F^{s-1}_{p,q}} ( \|u\|_{F^s_{p,q}}+\|v\|_{F^s_{p,q}} ).  \no 
\end{align}
In the same way, one can assert
\begin{align}
\|w\cdot \nabla u\|_{F^{s-1}_{p,q}} \leq C \|w\|_{F^{s-1}_{p,q}} \|u\|_{F^s_{p,q}}.  \no 
\end{align}
Lastly, by virtue of Proposition~\ref{esticomm1} and Proposition~\ref{esticomm2}, 
\begin{align}
\|2^{j(s-1)}[v,\Delta_j]\cdot\nabla w\| _{L_x^pl_j^q} &\leq C (\|\nabla v\|_{L^{\infty}} \|w\|_{\dot{F}^{s-1}_{p,q}} + \|w\|_{L^{\infty}} \|v\|_{\dot{F}^s_{p,q}})  \no \\
& \leq C \|w\|_{F^{s-1}_{p,q}} \|v\|_{F^s_{p,q}}.  \no 
\end{align}
Summarizing the estimates above,  we deduce
\begin{align} \label{bddfs1}
\|w(t)\|_{F^{s-1}_{p,q}} \leq \|w_0\|_{F^{s-1}_{p,q}} + C \int_0^t \|w(\tau)\|_{F^{s-1}_{p,q}} (\|u(\tau)\|_{F^s_{p,q}}+ \|v(\tau)\|_{F^s_{p,q}}) d\tau.
\end{align}
Using~\eqref{bddintri} and Gronwall's inequality, one can have
\begin{align}
\|w\|_{L_T^{\infty}F^{s-1}_{p,q}} \leq C \|w_0\|_{F^{s-1}_{p,q}}.  \no 
\end{align}
which justifies~\eqref{differ}.
\end{proof}

{\it Step 3.}  Let $u_0 \in D(R)$,  we claim 
\begin{align} \label{bonasmith}
\|S_{T}(u_0) - S_{T}(P_{\leq N}u_0)\|_{L_T^{\infty} F^s_{p,q}} \leq C \|u_0-P_{\leq N} u_0\|_{F^s_{p,q}}.
\end{align}
For simplicity,  let us denote $u = S_T({u_0}),\ u^{N} = S_{T}(P_{\leq N}u_0)$ and $w^{N} = u - u^{N}$,  according to the result in Step 1 and Remark~\ref{estimatelow}, there exists some $T=T(R)$, such that 
\begin{align}
\|u\|_{L_T^{\infty}F^s_{p,q}}+ \|u^N\|_{L_T^{\infty}F^s_{p,q}} & \leq C \|u_0\|_{F^s_{p,q}} \leq C,  \\
   \|u^N \|_{L_T^{\infty}F^{s+1}_{p,q}} & \leq C \|P_{\leq N}u_0\|_{F^{s+1}_{p,q}} \leq C 2^N.   \label{bddhigher}
\end{align}
It is not hard  to see $w^N$ is a solution of  the following equation:
\begin{equation*}
\begin{cases}
\partial_t w^N + u\cdot \nabla w^N + w^N \cdot \nabla u^N + \nabla (P(u) - P(u^N)) = 0 ,  & \\
{\rm div}\, w^N = 0, & \\
w^N(0,x) = w^N_0 = u_0 - P_{\leq N} u_0.
\end{cases}
\end{equation*}
By means of  argument similar  to that in  Step 2,    one can deduce 
\begin{equation} \label{estimate4}
\begin{split}
\|w^N(t)\|_{F^s_{p,q}} \leq \|w_0^N\|_{F^s_{p,q}} + \int_0^t \|w^N\cdot \nabla u^N\|_{F^s_{p,q}} + \|\nabla (P(u)-P(u^N))\|_{F^s_{p,q}}   \\ 
+ \big\|2^{js}[u,\Delta_j]\cdot \nabla w^N \big\|_{L_x^pl_j^q} d\tau.
\end{split}
\end{equation}
Following~\eqref{controlnonho}, Lemma~\ref{kernelinteg} and Propositions~\ref{nonendmoser} and~\ref{endmoser},  we can get
\begin{align} \label{estimate1}
\|w^N\cdot \nabla u^N\|_{F^s_{p,q}}& \leq C\|w^N\|_{L^{\infty}} \|u^N\|_{L^p}+C \|w^N\|_{L^{\infty}} \|\nabla u^N\|_{\dot{F}^s_{p,q}} + C\|w^N\|_{\dot{F}^s_{p,q}} \|\nabla u^N\|_{L^{\infty}}  \no \\
& \leq C2^N\|w^N\|_{F^{s-1}_{p,q}} +  C\|w^N\|_{F^s_{p,q}} \|u^N\|_{F^s_{p,q}}  \no \\
& \leq C \|w_0^N\|_{F^s_{p,q}}+ C\|w^N\|_{F^s_{p,q}}.
\end{align}
where we used the fact $2^N \|w_0^N\|_{F^{s-1}_{p,q}} \leq C \|w_0^N\|_{F^s_{p,q}}
$ in the last inequality. Regarding the pressure, 
\begin{align}
\nabla (P(u)-P(u^N)) = \nabla (-\Delta)^{-1}{\rm div} (w^N\cdot \nabla u^N + u \cdot \nabla w^N).  \no 
\end{align}
By Lemma~\ref{equvitri} and the boundedness of Riesz operator in $\dot{F}^{s-1}_{p,q}$, one can have
\begin{align}
& \| \nabla (-\Delta)^{-1}{\rm div} (w^N\cdot \nabla u^N)\|_{F^s_{p,q}} \no \\
 &\leq C\|w^N\|_{L^p} \| u^N\|_{L^{\infty}} + C\sum_{1\leq k \leq d}  \big\|(\partial_k w^N) \cdot  \nabla (u^N)^k \big\|_{\dot{F}^{s-1}_{p,q}}  \no \\
& \leq 	C\|w^N\|_{F^s_{p,q}} \|u^N\|_{F^s_{p,q}}.  \no 
\end{align}
where Propositions~\ref{nonendmoser}-\ref{endmoser} is used in the last inequality, $(u^N)^k$ denotes $k-$th component of $u^N$.  This yields
\begin{align} \label{estimate2}
\|\nabla (P(u)-P(u^N))\|_{F^s_{p,q}} \leq C \|w^N\|_{F^s_{p,q}}( \|u^N\|_{F^s_{p,q}}+ \|u\|_{F^s_{p,q}}).
\end{align}
Moreover, on account of Propositions~\ref{esticomm1}-\ref{esticomm2}, we find
\begin{align} \label{estimate3}
\|2^{js}[u,\Delta_j]\cdot \nabla w^N\|_{L_x^pl_j^q} \leq C  \|w^N\|_{F^s_{p,q}} \|u\|_{F^s_{p,q}}.
\end{align}
Thereby the estimates \eqref{estimate1}, \eqref{estimate2}-\eqref{estimate3} in conjunction with~\eqref{estimate4} and~\eqref{bddintri} can imply 
\begin{align}  \label{bddfs}
\|w^N(t)\|_{F^s_{p,q}} \leq C\|w_0^N\|_{F^s_{p,q}} + C\int_0^t \|w^N(\tau)\|_{F^s_{p,q}} d\tau, 
\end{align}
Applying  Gronwall's  inequality, we get
\begin{align}
\|w^N\|_{L_T^{\infty}F^s_{p,q}} \leq C \|w_0^N\|_{F^s_{p,q}}.  \no
\end{align}
from which~\eqref{bonasmith} follows.

{\it Step 4.} Based on the aforementioned estimates, we show  the continuity of the solution map. Let $\psi, u_0 \in D(R)$,  then
\begin{equation*}
\begin{split}
& \|S_{T}(u_0) -S_{T}(\psi)\|_{L_T^{\infty}F^s_{p,q}} \\
& \leq \|S_{T}(u_0) -S_{T}(P_{\leq N}u_0)\|_{L_T^{\infty}F^s_{p,q}} + 
\|S_{T}(\psi) -S_{T}(P_{\leq N}\psi)\|_{L_T^{\infty}F^s_{p,q}} \\
 &\qquad \qquad \qquad \qquad \qquad  + \|S_{T}(P_{\leq N}u_0) -S_{T}(P_{\leq N}\psi)\|_{L_T^{\infty}F^s_{p,q}}  \\
& \leq C\big( \|u_0 -P_{\leq N}u_0\|_{F^s_{p,q}} + \|\psi -P_{\leq N}\psi\|_{F^s_{p,q}} \big)  \\
& \qquad \qquad + C\|S_{T}(P_{\leq N}u_0) -S_{T}(P_{\leq N}\psi)\|^{\frac{1}{2}}_{L_T^{\infty}F^{s-1}_{p,q}} \|S_{T}(P_{\leq N}u_0) -S_{T}(P_{\leq N}\psi)\|^{\frac{1}{2}}_{L_T^{\infty}F^{s+1}_{p,q}}  \\
& \leq C\big( \|u_0 -P_{\leq N}u_0\|_{F^s_{p,q}} + \|\psi -u_0\|_{F^s_{p,q}} \big)  
+ C 2^{N/2}R^{1/2}\|u_0-\psi\|^{\frac{1}{2}}_{F^{s-1}_{p,q}}.
\end{split}                  
\end{equation*}
here we employed~\eqref{differ} and~\eqref{bddhigher} in the last inequality. 
As $1 \leq p,\,q < \infty$, so for arbitrary $\epsilon>0$, one can select $N$ to be sufficiently large, such that
\[
C\|u_0 -P_{\leq N}u_0\|_{F^s_{p,q}} \leq \frac{\epsilon 
}{2}. 
\]
Then fix $N$, choose $\delta$ so small that $\|u_0-\psi\|_{F^s_{p,q}} < \delta$ and $C\delta + CR^{1/2}2^{N/2} \delta^{1/2} < \epsilon/2$. Hence,
\[
 \|S_{T}(u_0) -S_{T}(\psi)\|_{L_T^{\infty}F^s_{p,q}}  < \epsilon.
\]
this concluded the proof.


\appendix 

\section{Local Cauchy theory for the Euler equations}

In this appendix, we state and briefly show the well-known local in time existence and uniqueness for the Euler equations in Triebel-Lizorkin spaces $F^s_{p,q}(\R^d)$. For a completed  treatment, one can refer to~\cite{Ch02,Ch03}. 
Recall that $D(R):=\{\phi \in F^s_{p,q} : \|\phi\|_{F^s_{p,q}} \leq R,\ {\rm{div}}\,\phi = 0\}$, 
the primary result of this part is as follows:
\begin{thm}  \label{lwp}
Let the space dimension $d\geq 2$ and $(s,p,q)$ be such that 
\begin{align}
s>1+\frac{d}{p},\ (p,q)\in (1,\infty)\times (1,\infty) \ \ \ \textrm{or} \ \ s \geq d+1,\ p =1,\  q\in [1,\infty). \no
\end{align}
Suppose $u_0 \in D(R) $, then there exists some time $T=T(R)>0$ and a unique solution $u\in C([0,T]; F^s_{p,q})$ to the Euler equations.
\end{thm}
\begin{proof}
As stated in Section~\ref{intro}, we will briefly outline the proof, as to the estimates involved, we omit the  reasoning arguments and just present the result, which can essentially be established by applying Propositions~\eqref{nonendmoser}-\eqref{endmoser} and Propositions~\eqref{esticomm1}-\eqref{esticomm2}, see also Section~\ref{pfmainresult}. Let $(u^{(m)}, p^{(m)})_{m\geq 0}$ be a sequence satisfying 
\begin{equation} \label{iterationequation}
\begin{cases}
\partial_t u^{(m)} + u^{(m-1)}\cdot \nabla u^{(m)} + \nabla p^{(m-1)} = 0, \\
{\rm div}\,u^{(m)} = 0, \\
u^{(m)}(x,0) = P_{\leq m} u_0.
\end{cases}
\end{equation} 
with $u^{(0)}=p^{(0)} = 0 $, $\|u_0\|_{F^s_{p,q}} \leq R$.  The proof can be divided into five steps:

{\it Step 1.} First we claim that $u^{(m)}$ is uniformly bounded for some small time.  Following argument that leads to~\eqref{bddfs}, one can assert
\begin{align}
\|u^{(m)}(t)\|_{F^s_{p,q}} \leq \|P_{\leq m}u_0\|_{F^s_{p,q}} + C \int_0^t \|u^{(m-1)}(\tau)\|_{F^s_{p,q}} \|u^{(m)}(\tau)\|_{F^s_{p,q}} d\tau. \no
\end{align}
Thus by Remark~\ref{estimatelow},
\begin{align}
\|u^{(m)}\|_{L^{\infty}_TF^s_{p,q}} \leq C \|u_0\|_{F^s_{p,q}} + CT\|u^{(m-1)}\|_{L^{\infty}_TF^s_{p,q}} \|u^{(m)}\|_{L^{\infty}_TF^s_{p,q}}.  \no 
\end{align}
Now we specify $T \leq \tilde{T}_0$ by taking $8C^2\tilde{T}_0\|u_0\|_{F^s_{p,q}} \leq 1$, then it follows by standard induction argument that
\begin{align} \label{sequbdd}
\|u^{(m)}\|_{L^{\infty}_{\tilde{T}_0}F^s_{p,q}} \leq 2C \|u_0\|_{F^s_{p,q}}, \ \ \forall\, m \geq 0.
\end{align}
Moreover
\begin{align} 
\|u^{(m)}\|_{L_{T}^{\infty} F^{s+1}_{p,q}} &\leq C \|P_{\leq m}u_0\|_{F^{s+1}_{p,q}} + CT\|u^{(m-1)}\|_{L^{\infty}_TF^s_{p,q}} \|u^{(m)}\|_{L^{\infty}_TF^{s+1}_{p,q}}  \no \\
& \qquad \qquad \qquad + CT \|u^{(m-1)}\|_{L^{\infty}_TF^{s+1}_{p,q}} \|u^{(m)}\|_{L^{\infty}_TF^{s}_{p,q}} \no \\
&\leq   C 2^{m}\|u_0\|_{F^s_{p,q}} + CTR \big( \|u^{(m-1)}\|_{L^{\infty}_TF^{s+1}_{p,q}}+ \|u^{(m)}\|_{L^{\infty}_TF^{s+1}_{p,q}} \big). \no 
\end{align}
Iterating again, one can find some ${T}_0 ={T_0}(R) \leq \tilde{T}_0$, such that 
\begin{align} \label{higherorder}
\|u^{(m)}\|_{L_{{T_0}}^{\infty} F^{s+1}_{p,q}} \leq 4C2^m \|u_0\|_{F^s_{p,q}}. 
\end{align}
Since $u^{(m)}$ also solves the following integral equation(Duhamel formula)
\begin{align}
u^{(m)}(x,t) = P_{\leq m} u_0 + \int_{0}^t \mathbb{P} \big(u^{(m-1)} \cdot \nabla u^{(m)}  \big)(\tau) d\tau. \no 
\end{align}
where $\mathbb{P}:= {\rm Id }- \nabla \Delta^{-1} {\rm div} $ is the Leray projector operator onto divergence free vector field.
We readily see 
\begin{align*}
& \|u^{(m)} (t_1)-u^{(m)} (t_2) \|_{F^s_{p,q}}  \no \\
& \leq \bigg| \int_{t_1}^{t_2} \big \| \mathbb{P} \big(u^{(m-1)} \cdot \nabla u^{(m)}  \big)   \big \|_{F^s_{p,q}}  d\tau  \bigg |  \no \\
& \leq  C\bigg| \int_{t_1}^{t_2}  \|  u^{(m-1)}  \|_{F^s_{p,q}}  \| u^{(m)}  \|_{F^s_{p,q}}  +  \|  u^{(m-1)}  \|_{F^{s-1}_{p,q}}  \| u^{(m)}  \|_{F^{s+1}_{p,q}}  d\tau  \bigg |  \no \\
& \leq C2^m R^2 \ |t_1 -t_2|.
\end{align*}
where we used~\eqref{sequbdd} and \eqref{higherorder}  in the last step. This infers that for each fixed $m \geq 0$, $u^{(m)} \in C([0,{T}_0]; F^s_{p,q})$. 

{\it Step 2.}  Let $v_0 \in D(cR)$ for some universal constant $c$,  $\{v^{(m)}, q^{(m)} \}_{m \geq 0}$ solves~\eqref{iterationequation} with initial data $P_{\leq m}v_0$, we claim that there exist some $T_1= T_1(R)$ and a constant $C$ independent of $m$,  such that 
\begin{align} \label{differestimate}
\|u^{(m)} - v^{(m)}\|_{L_{T_1}^{\infty} F^{s-1}_{p,q}} \leq C \| u_0 -  v_0\|_{F^{s-1}_{p,q}}, \ \ \  \forall\ m \geq 0. 
\end{align}
Indeed,  according to results in Step 1,  one can say there exists some $\tilde{T}_1 = \tilde{T}_1 (R)$,  s.t.
\begin{align} \label{estimateapp}
\|u^{(m)}\|_{L_{\tilde{T}_1}^{\infty}F^s_{p,q} } \leq C \|u_0\|_{F^s_{p,q}}, \ \ \  \|v^{(m)}\|_{L_{\tilde{T}_1}^{\infty}F^s_{p,q} } \leq C \|v_0\|_{F^s_{p,q}}, \ \ \forall\ m\geq 0.
\end{align}
Now set $w^{(m)} = u^{(m)} -v^{(m)}$,  $w_0 = u_0 -v_0$,  then 
\begin{equation*}
\begin{cases}
\partial_t w^{(m)} + u^{(m-1)}\cdot \nabla w^{(m)} + w^{(m-1)}\cdot \nabla v^{(m)} + \nabla (p^{(m-1)} -q^{(m-1)})= 0, \\
{\rm div}\,w^{(m)} = 0, \\
w^{(m)}(x,0) = w_0^{(m)} = P_{\leq m} w_0.
\end{cases}
\end{equation*}
Hence,  
\begin{align*}
\|w^{(m)}(t)\|_{F^{s-1}_{p,q}}& \leq \|w_0^{(m)}\|_{F^{s-1}_{p,q}} + \int_{0}^t \|w^{(m-1)}\cdot \nabla v^{(m)}\|_{F^{s-1}_{p,q}} + \|\nabla (p^{(m-1)} -q^{(m-1)}) \|_{F^{s-1}_{p,q}}   \\
& \qquad \qquad \qquad+  \big\|2^{j(s-1)} [u^{(m-1)}, \Delta_j ]\cdot \nabla w^{(m)} \big\|_{L_x^p l_j^q}  d\tau  \\
& \leq \|w_0^{(m)}\|_{F^{s-1}_{p,q}} +C \int_0^t \|w^{(m-1)}\|_{F^{s-1}_{p,q}} \|v^{(m)}\|_{F^{s}_{p,q}}   \\
& \qquad \qquad \qquad \qquad \qquad + \|w^{(m)}\|_{F^{s-1}_{p,q}} \|u^{(m-1)}\|_{F^{s}_{p,q}}  d\tau.
\end{align*}
Applying~\eqref{estimateapp} with some $T_1 \leq \tilde{T}_1$,  we have 
\begin{align}
\|w^{(m)}(t)\|_{L_{T_1}^{\infty}F^{s-1}_{p,q}}  \leq C\|w_0\|_{F^{s-1}_{p,q}} +CT_1R \big(  \|w^{(m-1)}\|_{L_{T_1}^{\infty}F^{s-1}_{p,q}}  + \|w^{(m)}\|_{L_{T_1}^{\infty}F^{s-1}_{p,q}} \big)   \no 
\end{align}
Selecting $T_1$ so small that $CT_1R \leq 1/8$, by an induction argument, one can  immediately see
\begin{align}
\|w^{(m)}\|_{L_{T_1}^{\infty}F^{s-1}_{p,q}}  \leq 2C \|w_0\|_{F^{s-1}_{p,q}} , \ \ \ \ \forall\ m \geq 0.  \no 
\end{align}
This yields~\eqref{differestimate}.

{\it Step 3.}  Let $u^{(m,N)}$  solves equation~\eqref{iterationequation} with initial data $P_{\leq m} P_{\leq N} u_0$, i.e. 
\begin{equation*}
\begin{cases}
\partial_t u^{(m,N)} + u^{(m-1,N)}\cdot \nabla u^{(m,N)} + \nabla p^{(m-1,N)} = 0, \\
{\rm div}\,u^{(m,N)} = 0, \\
u^{(m,N)}(x,0) = P_{\leq m} P_{\leq N}u_0.
\end{cases}
\end{equation*}
where $u^{(0,N)} = p^{(0,N)} = 0$.  Then  there exist some $T_2 = T_2(R)$ and a constant $C$ independent of $m, N$,  satisfying 
\begin{align}
\| u^{(m)} - u^{(m,N)}\|_{L_{T_2}^{\infty}F^{s}_{p,q}}  \leq C \|u_0 - P_{\leq N} u_0\|_{F^{s}_{p,q}}. \no 
\end{align}
We can argue as follows:  by the estimates in Step 1, $ \exists\, \tilde{T}_2 = \tilde{T}_2(R) \leq T_0$, 
\begin{align}
\|u^{(m,N)}\|_{L^{\infty}_{\tilde{T}_2}  F^{s+k}_{p,q}} \leq  C2^{Nk} \|u_0\|_{F^s_{p,q}}, \ \ \ \forall\, m \geq 0,\, k =0,\,1.  \no 
\end{align}
Now let $w^{(m,N)} = u^{(m)}- u^{(m,N)}$,  $w_0^{(N)}:=u_0 - P_{\leq N} u_0 $, then 
\begin{equation*}
\begin{cases}
\partial_t w^{(m,N)} + u^{(m-1)}\cdot \nabla w^{(m,N)} + w^{(m-1,N)}\cdot \nabla u^{(m,N)} + \nabla (p^{(m-1)} -p^{(m-1,N)})= 0, \\
{\rm div}\,w^{(m,N)} = 0, \\
w^{(m,N)}(x,0) = w_0^{(m,N)} = P_{\leq m} w_0^{(N)}.
\end{cases}
\end{equation*}
Similarly, for $0< t \leq T_2 \leq \tilde{T}_2$,
\begin{align*}
& \quad  \|w^{(m,N)}(t)\|_{F^{s}_{p,q}}  \\
& \leq \|w_0^{(m,N)}\|_{F^{s}_{p,q}} + \int_{0}^t C 2^N R \|w^{(m-1,N)}\|_{F^{s-1}_{p,q}}    +  C R  \|w^{(m,N)}\|_{F^{s}_{p,q}}   +  C R \|w^{(m-1,N)}\|_{F^{s}_{p,q}} d\tau.   
\end{align*}
Noticing formula~\eqref{differestimate},  we obtain
\begin{align*} 
\|w^{(m,N)}\|_{L_{T_2}^{\infty}F^{s}_{p,q}}   \leq C\|w_0^{(N)}\|_{F^{s}_{p,q}} +  C R T_2\big(   \|w^{(m-1,N)}\|_{L_{T_2}^{\infty}F^{s}_{p,q}}  + \|w^{(m,N)}\|_{L_{T_2}^{\infty}F^{s}_{p,q}}  \big ).
\end{align*}
Now choosing $CT_2 R $ small enough and performing an induction on $m$, we have
\begin{align}
\|w^{(m,N)}\|_{L_{T_2}^{\infty}F^{s}_{p,q}} \leq 2C\|w_0^{(N)}\|_{F^{s}_{p,q}}, \ \ \ \forall\, m,\, N \geq 0. \no 
\end{align} 
The desired result then follows.
 
{\it Step 4.} Next we show $\{u^{(m)}\}$ is a Cauchy sequence in $C([0,T_*]; F^{s-1}_{p,q})$ for some $T_* =T_*(R) \leq {\rm min}\{ T_0, T_2\} = \bar{T}$. In fact, set $w^{(m)} = u^{(m)}- u^{(m-1)}$, one can easily see $w^{(m+1)}$ satisfies 
\begin{equation*}
\begin{cases}
\partial_t w^{(m+1)} + u^{(m)}\cdot \nabla w^{(m+1)} + w^{(m)}\cdot \nabla u^{(m)} + \nabla (p^{(m)} - p^{(m-1)}) = 0,  \\
{\rm div}\, w^{(m+1)} = 0,  \\
w^{m+1}(0,x)= \Delta_{m+1} u_0.
\end{cases}
\end{equation*}
Likewise,  
\begin{align*}
\|w^{(m+1)}(t)\|_{F^{s-1}_{p,q}} \leq \|\Delta_{m+1} u_0\|_{F^{s-1}_{p,q}} + C\int_0^t \|u^{(m)}\|_{F^{s}_{p,q}} \big(\|w^{(m)}\|_{F^{s-1}_{p,q}}+\|w^{(m+1)}\|_{F^{s-1}_{p,q}} \big)d\tau.
\end{align*}
By~\eqref{sequbdd},  we get
\begin{equation*}
\begin{split}
\|w^{(m+1)}\|_{L_{T_*}^{\infty}F^{s-1}_{p,q}} \leq C 2^{-(m+1)} \|u_0\|_{F^s_{p,q}} & + CT_*R
\| w^{(m+1)}\|_{L_{T_*}^{\infty}F^{s-1}_{p,q}}  \\
& +CT_* R \|w^{(m)}\|_{L_{T_*}^{\infty}F^{s-1}_{p,q}}.
\end{split}
\end{equation*}
Now choosing $T_* \leq \bar{T}$, such that $16C^2T_* R \leq 1/2$, by a simple iteration, one can show
$$
\|w^{(m)}\|_{L_{T_*}^{\infty}F^{s-1}_{p,q}} \leq (4C)2^{-m}\|u_0\|_{F^s_{p,q}}, \ \ \forall\ m\geq 0.
$$
This exponential decay implies what we want. 

{\it Step 5.} Finally we prove $\{u^{(m)}\}_{m\geq 0}$ is a Cauchy sequence in $X_{T_*}^s = C([0,T_*]; F^s_{p,q})$. 
According to the conclusion in Step 4, one can also  claim that $\{u^{(m,N)}\}_{m \geq 0}$ is a Cauchy sequence  with 
\begin{align}
\|u^{(m,N)}-u^{(n,N)}\|_{L_{T_*}^{\infty}F^{s-1}_{p,q}} \leq C 2^{-m} R, \ \ \ \ n\geq m, \ \forall\, N \geq 0, \no 
\end{align} 
here $C$ doesn't depend on $m, n, N$. As a consequence, for $n \geq m$, 
\begin{align}
&\|u^{(m)} - u^{(n)}\|_{L^{\infty}_{T_*} F^s_{p,q}}  \no \\
&\leq \|u^{(m)} - u^{(m,N)}\|_{L^{\infty}_{T_*} F^s_{p,q}} + \|u^{(m,N)} - u^{(n,N)}\|_{L^{\infty}_{T_*} F^s_{p,q}} + \|u^{(n,N)} - u^{(n)}\|_{L^{\infty}_{T_*} F^s_{p,q}}  \no \\
& \leq C \|u_0 - P_{\leq N} u_0\|_{F^s_{p,q}} + C\|u^{(m,N)} - u^{(n,N)}\|^{\frac{1}{2}}_{L^{\infty}_{T_*} F^{s-1}_{p,q}} \|u^{(m,N)} - u^{(n,N)}\|^{\frac{1}{2}}_{L^{\infty}_{T_*} F^{s+1}_{p,q}}  \no  \\
& \leq C \|u_0 - P_{\leq N} u_0\|_{F^s_{p,q}} + C 2^{-m/2} 2^{N/2} R.   \no 
\end{align}
Since $1 \leq p, q < \infty$, the Schwartz function is dense in $F^s_{p,q}$, see~\cite{Tr83}, one can assert that the first term $C \|u_0 - P_{\leq N} u_0\|_{F^s_{p,q}}$ can be made arbitrarily small provided that $N$ is large enough. Then fix such $N$, taking $m$ to be sufficiently large, the second term  $C 2^{-m/2} 2^{N/2} R $ can also be as small as we want, so 
$\{u^{(m)}\}_{m \geq 0}$ is a Cauchy sequence in $X_{T_*}^s$ and converges to some $u\in X_{T_*}^s$. 
In view of~\eqref{iterationequation}, we find that the limit $u$ is a solution of the Euler system with initial data $u_0 \in F^s_{p,q}$ and meets~\eqref{sequbdd} as well. This finishes the local existence of solution in $X_{T_*}^s$, as to the uniqueness, which essentially can be done in the same way as the estimate in Step 4, we refer reader to~\cite{Ch02} for more details.
\end{proof}

\subsection*{Acknowledgements}
Z.~Guo  is partially supported by ARC DP170101060.  


\end{document}